\documentclass[12pt,english]{smfart}
\usepackage[T1]{fontenc}
\usepackage[english,francais]{babel}
\usepackage{color}
\usepackage[
    colorlinks=true,
    citecolor=red
]{hyperref}
\usepackage{mathrsfs}
\usepackage{amssymb,url,xspace,smfthm}


\newtheorem{theorem}{Theorem}[section]
\newtheorem{proposition}{Proposition}[section]

\newtheorem{lemma}{Lemma}[section]
\newtheorem{corollary}{Corollary}[section]
\newtheorem{definition}{Definition}[section]

\numberwithin{equation}{section} \numberwithin{theorem}{section}
\numberwithin{proposition}{section} \numberwithin{lemma}{section}
\numberwithin{corollary}{section}
\numberwithin{definition}{section} \numberwithin{remark}{section}

\newcommand{\BibTeX}{{\scshape Bib}\kern-.08em\TeX}
\newcommand{\T}{\S\kern .15em\relax }
\newcommand{\AMS}{$\mathcal{A}$\kern-.1667em\lower.5ex\hbox
        {$\mathcal{M}$}\kern-.125em$\mathcal{S}$}

\newcommand{\R}{\mathbb{R}}

\title{Stable domains for higher order elliptic operators}
\date {Version 1, mai 2023}
 \author{Jean-Fran\c{c}ois Grosjean, Antoine Lemenant and R\'emy Mougenot}

\address{Université de Lorraine, CNRS, IECL, F-54000 Nancy, France}

\email{jean-francois.grosjean@univ-lorraine.fr, antoine.lemenant@univ-lorraine.fr, remy.mougenot@univ-lorraine.fr}
\keywords{Capacity, stable, $\gamma_m-$convergence, shape optimisation}

\date{\today}

\begin{document}

\maketitle
 
\begin{abstract} This paper is devoted to prove that any domain satisfying a $(\delta_0,r_0)-$capacity condition of first order is automatically $(m,p)-$stable  for all $m\geqslant 1$ and $p\geqslant 1$, and for any dimension $N\geqslant 1$. In particular, this includes regular enough domains such as $\mathscr{C}^1-$domains, Lipchitz domains, Reifenberg flat domains, but is weak enough to also includes cusp points. Our result  extends some of the results of Hayouni and Pierre   valid only for $N=2,3$, and extends also the results of Bucur and Zolesio for higher order operators, with a different and simpler proof.  
\end{abstract}

 \vspace{1cm}

\tableofcontents

\newpage

\section{Introduction}

Let $\Omega \subset \R^N$ be a bounded and open set. Following \cite{HePi06,BuBut05,BuZo94}, we say that $\Omega$ is $(m,p)-$stable if 
$$W^{m,p}(\R^N) \cap \left\{u =0 \text{ a.e. in } \overline{\Omega}^c    \right\} = W^{m,p}_0(\Omega).$$
This notion  is related to the continuity of a  $2m-$order elliptic PDE with respect to  domain perturbation. In particular, if $\Omega$ is $(m,2)-$stable, then it implies that for any sequence of domains $(\Omega_n)_{n\in \mathbb{N}}$ converging to $\Omega$ in a certain  Hausdorff sense, one has that $(u_n)_{n\in \mathbb{N}}$ converges strongly in $H^{m}$ to $u$, where $u_n$ is the  unique solution in $H^m_{0}(\Omega_n)$ for the equation $(-\Delta)^m(u_n)=f$ in $\Omega_n$, and $u$ is the solution of the same problem in $\Omega$. It is also equivalent to the convergence of $(W^{m,p}_0(\Omega_n))_{n\in \mathbb{N}}$ to $W^{m,p}_0(\Omega)$ in the sense of Mosco (see Section \ref{stability}).

In the literature, a lot of attention has been devoted to the case $m=1$ and $p=2$ because of its relation to the Laplace operator. On the other hand, very few results are available for the higher order spaces $H^m_0(\Omega)$, related to bi-harmonic or more generally poly-harmonic equations, that have a lot of applications.  The objective of this paper is to give a short and elementary proof of the fact that any domain which is ``regular enough'' is always $(m,p)-$stable for all $m,p$ and all dimensions $N$.  

 Notice that in general, the stability for $W^{m,p}_0(\Omega)$ does not simply reduce to the one of $W^{1,p}_0(\Omega)$. To enlight this fact we recall that for every open set $\Omega \subset \R^N$, we have the characterisation (see for instance \cite{AdHe96})
$$W^{m,p}_0(\Omega)= W^{m,p}(\R^N)\cap \left\{\nabla^ku|_{\Omega^c}=0 \; \quad (m-k,p)-\text{q.e.  for all }   k\leqslant m-1  \; \right\},$$
where $\nabla^k u := (\partial^\alpha u)_{\vert \alpha \vert = k }$ and $\partial^\alpha u$ is the  $(m-k,p)-$quasicontinuous representative, which is in particular defined pointwise $(m-k,p)-$q.e.
If $\Omega$ is $(1,p)-$stable, then for any $|\alpha|\leqslant m-1$ and from the assumption $\partial^\alpha u=0$ a.e. in $\overline{\Omega}^c$ we would only deduce that $\partial^\alpha u=0$ $(1,p)-$q.e. on $\Omega^c$,  whereas in order to prove that $u\in W^{m,p}_0(\Omega)$ we would need the stronger condition $\partial^\alpha u=0$ $(m-|\alpha|,p)-$q.e. on $\Omega^c$.

In \cite{HaPi01}, Hayouni and Pierre exploited the compact embedding of  $H^2$  into continuous functions in dimensions 2 and 3, in order to get some stability results for the space $H^2_0$. In particular, they proved that, in dimension 2 and 3, any $(1,2)-$stable domain is automatically $(2,2)-$stable (see \cite{HaPi01} or \cite{HePi06}). They also proved in the same paper that, in dimensions 2 and 3, any sufficiently smooth domain is a $(2,2)-$stable domain.  

In the present paper, we show that there is no true restriction on dimension $N$ to obtain  $(m,p)-$stability.  Our main result says  that  any domain that satisfies a classical $(1,p)-$ capacitary condition will be automatically $(m,p)-$stable, in any dimension, and for any $m$. This includes a large class or ``regular'' domains such as $\mathscr{C}^1-$domains, Lipschitz domains,  Reifenberg-flat domains, domains satisfying the so-called external corkscrew condition (see Definition \ref{corkscrew}), $\varepsilon$-cone property, or even domains with segment property which allows domains with cusps, or more generally domains with the so called fat cone property \cite{bz2}.

For the rest of this paper, we restrict ourselves to open subset of a fixed ball $D\subset \mathbb{R}^N$, and we denote the set of admissible domains by
$$ \mathscr{O}(D) := \left\{ \Omega \; \middle| \; \Omega \subseteq D \text{ is open}\right\}.$$\newline

\begin{definition}
    Let $r_0>0$ and $\delta_0>0$. An open set $\Omega\subseteq \mathbb{R}^N$ has the $(r_0, \delta_0)-$capacitary condition if for all $x\in \partial\Omega$ and for all $r\leqslant r_0$,
    \begin{eqnarray}
\frac{{\rm Cap}_{1,p}(\overline{\Omega}^c\cap B(x,r))}{{\rm Cap}_{1,p}( B(x,r))}\geqslant \delta_0\label{cond}.
\end{eqnarray}
The class of open subset of $D$ having the $(r_0, \delta_0)-$capacity condition is denoted by $\mathscr{O}^{\delta_0,r_0}_{\textnormal{cap}}(D)$.
\end{definition}\leavevmode

Here is our main statement.\newline

\begin{theorem}\label{main}If $\Omega\in\mathscr{O}^{\delta_0,r_0}_{\textnormal{cap}}(D)$ satisfies $|\partial  \Omega|=0$, then $\Omega$ is $(m,p)-$stable for any $m\geqslant 1$ and $1\leqslant p <+\infty$.
\end{theorem}\leavevmode

Let us give some comments about the result.  One of the main feature and somewhat surprising is that the condition involves only the $(1,p)-$capacity even if the conclusion yields $(m,p)-$stability for all $m\geqslant 1$. In \cite{BuZo94}, Bucur and Zolesio proved that a domain is $(1,2)-$stable under a very similar but weaker condition with $(1,2)-$capacity. More precisely, the condition in  \cite{BuZo94} is the same as ours but without a bar over $\Omega$ in the numerator (See Section~\ref{sec7} for more details).  In contrast, with the very similar and slightly stronger $(1,2)-$capacity condition \eqref{cond}, we obtain $(m,2)-$stability for all $m\geqslant 1$. It is worth mentioning that our proof is different and much simpler than the one \cite{BuZo94}, thus provides an alternative argument which is new even for the standard case $m=1$.

As a consequence of our main result we get a capacitary condition which implies stability for the polyharmonic equation along a Hausdorff converging sequence of domains. We refer to Section \ref{stability}  for the definition of Hausdorff convergence, Mosco convergence and $\gamma_m$-convergence, and we give here in the introduction two different statements.  In the first one (Corollary \ref{main22}) we assume only the limiting domain $\Omega$ to be ``regular'' while in the second (Theorem \ref{main3}) we assume the whole sequence to be ``regular''. \newline

\begin{corollary} \label{main22} Let $\Omega\in\mathscr{O}^{\delta_0,r_0}_{\textnormal{cap}}(D)$ and $(\Omega_n)_{n\in \mathbb{N}}$ be a sequence in $\mathscr{O}(D)$. If $|\partial  \Omega|=0$, $(\overline{\Omega_n})_{n\in \mathbb{N}}$ $d_H-$converges to $ \overline{\Omega}$, and $(\Omega_n)_{n\in \mathbb{N}}$ $d_{H^c}-$converges to $\Omega$,
   then the sequence $(\Omega_n)_{n\in \mathbb{N}}$ $\gamma_m-$converges to $\Omega$, or equivalently, $(H^m_0(\Omega_n))_{n\in \mathbb{N}}$ converges to $H^m_0(\Omega)$ in the sense of Mosco.
\end{corollary}\leavevmode

Corollary \ref{main22} follows from gathering together Proposition \ref{main2} and Theorem \ref{main}. Let us now mention a few remarks.

\begin{enumerate}
\item  The interesting feature of Corollary \ref{main22} is that only the limiting domain $\Omega$ is assumed to be stable (thus somehow ``regular'') and nothing is assumed on the sequence $(\Omega_n)_{n\in \mathbb{N}}$, which could be arbitrary open sets. 

\item It is worth mentioning that in \cite{BuZo94} the authors assumed only $\Omega_n\overset{d_{H^c}}{\longrightarrow}  \Omega$ to obtain the $\gamma_m-$convergence of a sequence $(\Omega_n)_{n\in \mathbb{N}}$. On the other hand they assumed that every term $\Omega_n$ along the sequence satisfies a capacitary condition with uniform constants. A similar statement will be given later in Theorem~\ref{main3}.

\item  It is easy to construct an example of stable domain $\Omega$  (even smooth) and a sequence $(\Omega_n)_{n\in \mathbb{N}}$ such that $\Omega_n\overset{d_{H^c}}{\longrightarrow}  \Omega$ and $(\Omega_n)_{n\in \mathbb{N}}$ does not $\gamma_m-$converges to $\Omega$. This shows that without any other assumption on the sequence, the second assumption $\overline{\Omega_n} \overset{d_H}{\longrightarrow}  \overline{\Omega}$ is pivotal for the result to hold true. The construction is rather classical : consider the sequence made from an enumeration $x_i \in B(0,1)$ of points with rational coordinates. 
 Then define
$$\Omega_n :=B(0,2)\setminus \bigcup_{i=0}^n \{x_i\}.$$
It is easy to see that $(\Omega_n)_{n\in \mathbb{N}}$ converges to $\Omega:= B(0,2)\setminus \overline{B}(0,1)$ for the complementary Hausdorff distance, which is clearly a $(m,2)-$stable domain because the boundary is smooth. On the other hand, for dimension $N\geqslant2m$ we know that $\textnormal{Cap}_{m,2}(\{x_i\})=0$, so it is classical that $(\Omega_n)_{n\in \mathbb{N}}$ does not $\gamma_m-$converge to $\Omega$ (see \cite[Section 3.2.6, page 80]{HePi06} for the case $m=1$). On the other hand $\overline{\Omega_n}=\overline{B}(0,2)$ clearly does not Hausdorff converge to $\overline{\Omega}=\overline{B}(0,2)\setminus B(0,1)$, which explains why  Theorem \ref{main2} does not  apply.
\end{enumerate}

Next, in order to get  existence of shape optimisation problems for higher order equations under geometrical constraints, the following variant is  more usefull. Notice that here we suppose \eqref{cond} on the whole sequence and by this way we can avoid the Hausdorff convergence but only complementary Hausdorff convergence is enough.\newline

\begin{theorem} \label{main3}  Let $\Omega\in\mathscr{O}(D)$ and $(\Omega_n)_{n\in \mathbb{N}}$ all belonging   to $\mathscr{O}^{\delta_0,r_0}_\textnormal{cap}(D)$. If $\vert \partial \Omega \vert = 0 $ and 
$(\Omega_n)_{n\in \mathbb{N}}$ $d_{H^c}-$converges to $ \Omega$, then $(\Omega_n)_{n\in \mathbb{N}}$ $\gamma_m-$converges to $\Omega$, or equivalently, $(H^m_0(\Omega_n))_{n\in \mathbb{N}}$ converges to $H^m_0(\Omega)$ in the sense of Mosco.
\end{theorem}

\vspace{0.5cm}

Since complementary Hausdorff topology is relatively compact, it is easy to get existence results for shape optimisation problems using Theorem \ref{main3}, with additional geometrical constraints on the domain. This applies to various standard classes of domains such as uniformly Lipschitz domains, Reifenberg-flat, corkscrew, or $\varepsilon-$cone, as described in the last section of the paper (see Theorem \ref{shapeopt}).


\section{Preliminaries}

The term domain and the symbol $\Omega$ will be reserved for an open and bounded set in the $N-$dimensional euclidean space $\mathbb{R}^N$. The norm of a point $x\in \mathbb{R}^N$ is denoted by $\vert x \vert := (\sum_{i=1}^Nx_i^2)^{1/2}$. If $\alpha$ is a multi-indice, i.e. $\alpha \in \mathbb{N}^N$, then the norm of $\alpha$ is $\vert \alpha \vert := \sum_{i=1}^N\alpha_i$ and we define the partial derivative operator
\begin{align*}
    \partial^\alpha := \frac{\partial^{\vert \alpha\vert}}{\partial_1^{\alpha_1} \cdots \partial^{\alpha_k}_N},
\end{align*}
and the vector $\nabla^k:=(\partial^\alpha )_{\vert\alpha \vert = k}$.
The notations $\partial \Omega$ and $\overline{\Omega}$ stand for
the boundary and the closure of $\Omega$, respectively. Let $\mathscr{C}^\infty_c(\Omega)$ be the space of smooth functions with compact support in $\Omega$. The ball of radius $r\geqslant 0$ and centered at $x\in \mathbb{R}^N$ is denoted by $B(x,r)$. For $m\in \mathbb{N}$ and $p \in [1, +\infty[$, we consider the usual Sobolev space $W^{m,p}(\Omega)$ endowed with the norm
$$ \Vert u \Vert_{W^{m,p}(\Omega)}:= \left( \sum_{k=0}^m \Vert \nabla^k u\Vert_{L^p(\Omega)}^p \right)^{1/p},$$
where
$$ \Vert \nabla^k u\Vert_{L^p(\Omega)}^p := \int_\Omega \vert \nabla^k u \vert^p \;dx.$$
Finally, the space $W_0^{m,p}(\Omega)$ is the completion of $\mathscr{C}^\infty_c(\Omega)$ with respect to the norm \newline$\Vert \cdot \Vert_{W^{m,p}(\Omega)}$. 

When the dimension $N < mp$, elements of $W^{m,p}(\R^N)$ can be represented as continuous functions. However, if $N \geqslant mp$, this is no longer the case and the natural way of measuring by how much the functions deviate from continuity is by means of capacity. If $K \subset \R^N$ is a compact, then we define the $(m,p)-$capacity of $K$ by
$$ \textnormal{Cap}_{m,p}(K):= \inf\left\{ \Vert \varphi \Vert^p_{W^{m,p}} \; \middle| \; \varphi \in \mathscr{C}^\infty_c(\R^N) \textnormal{ such that }\varphi \geqslant 1 \textnormal{ on } K \right\}.$$
Afterwards, for an open set $\Omega \subseteq \R^N$ we can consider
$$ \textnormal{Cap}_{m,p}(\Omega):= \sup\left\{ \textnormal{Cap}_{m,p}(K) \; \middle| \; K \subseteq \Omega \textnormal{ is a compact} \right\}.$$
Finally, if $E \subseteq \R^N$ is an abritrary set, then we define
$$ \textnormal{Cap}_{m,p}(E):= \inf\left\{ \textnormal{Cap}_{m,p}(\Omega) \; \middle| \; \Omega \supseteq E \textnormal{ is a open set} \right\}.$$
When an assertion is true except for a set of $(m,p)-$capacity equal to zero, we say that it is true $(m,p)-$quasi everywhere and denote by $(m,p)-$q.e.. Let $\rho \in \mathscr{C}^\infty_c(B(0,1))$ be a test function and consider an approximate identity $(\rho_n)_{n\in \mathbb{N}}$ by $\rho_n(x):= n^N\rho(nx)$. For $(m,p)-$q.e. $x \in \mathbb{R}^N$, $\lim  \rho_n * u(x) := \Tilde{u}(x)$ exists and $u(x)=\Tilde{u}(x)$ almost everywhere. Moreover, for all $\varepsilon>0$ there is an open set $\Omega_\varepsilon\subset \mathbb{R}^N$ such that $\textnormal{Cap}_{m,p}(\Omega_\varepsilon)< \varepsilon$ and, for a subsequence, $\rho_n * u$ converges uniformly to $\Tilde{u}$ on $\mathbb{R}^N \backslash \Omega_\varepsilon$. In particular, the function $\Tilde{u}$ is continuous on $\mathbb{R}^N \backslash \Omega_\varepsilon$ and we call it the $(m,p)-$quasicontinuous representative of $u$. In the present paper, functions $u$ in $W^{m,p}(\mathbb{R}^N)$ will be assumed to be defined pointwise $(m,p)-$q.e. and to be $(m,p)-$quasicontinuous (see \cite{AdHe96}).


The proof of the main result will use  the following Poincar\'e type inequality that can be found for instance in \cite[Corollary 4.5.2, page 195]{Zi89}. To be more precise, we can apply the inequality in \cite{Zi89} to $B(0,1)$ and use the fact that $\textnormal{Cap}_{1,p}$ is homogeneous of degree $N-p$ (see \cite[Theorem 2, page 151]{EvGa92}), then by a simple change of variable apply it to the function $x\longmapsto u(Rx)$, we get the following one.\newline 

\begin{lemma}  Let $r>0$, and $u \in W^{1,p}(B(0,r))$. We define $Z(u):=\{ x \in \overline{B}(x_0,r) \; |\; u(x)=0\}$.
 If $\textnormal{Cap}_{1,p}(Z(u))>0$, then
 \begin{eqnarray}
\int_{B(0,r)} \vert u\vert^p \;dx \leqslant C \frac{r^p}{{\rm Cap}_{1,p}(r^{-1}Z(u))}  \int_{B(0,r)} |\nabla u|^p \;dx, \label{poincare}
\end{eqnarray}
where $C>0$ depends only on $p$ and $N$.
\end{lemma}\leavevmode

\section{Proof of Theorem \ref{main}}

\begin{proof}[Proof of Theorem \ref{main}] Let $\Omega$ be a bounded domain satisfying the assumptions of Theorem~\ref{main} and let $u\in W^{m,p}(\R^N)$ be given satisfying $u=0$ almost everywhere in $\overline{\Omega}^c$. To prove the theorem it suffice to prove that $u$ can be approximated in the $W^{m,p}(\R^N)$ norm by a sequence of functions in $\mathscr{C}^\infty_c(\Omega)$. To do so we will first truncate $u$ near the boundary of $\Omega$ as follows. For all $n \in \mathbb{N}$, we consider 
$$ K_n := \left\{ x \in \Omega \; \middle| \; d(x,\partial \Omega)\geqslant 2^{-n} \right\}.$$
The exhaustive family of compact $(K_n)_{n\in \mathbb{N}}$ satisfies $K_{n} \subseteq K_{n+1}$ and $\Omega = \bigcup_{n\in \mathbb{N}}K_n$.
Then take a test function $\rho \in \mathscr{C}^\infty_c(B(0,1))$ such that $\rho \geqslant 0$ and
$$ \int_{\R^N}\rho(x)dx = 1.$$
We define $\rho_\varepsilon(x):=\varepsilon^{-N} \rho(x/\varepsilon)$ and 
$$
\theta_{n,\varepsilon}(x) := {\bf 1}_{K_n}*\rho_{\varepsilon}(x) =\varepsilon^{-N} \int_{K_n}\rho\left(\frac{x-y}{\varepsilon}\right)dy,
$$
which satisfies $\textnormal{Supp }(\theta_{n,\varepsilon}) \subseteq K_n+ \overline{B}(0,\varepsilon)$.  We take $\varepsilon_n :=2^{n+1}$ and denote now $\theta_{n}:=\theta_{n,\varepsilon_n}$ in such a way that  $\theta_n\in \mathscr{C}^\infty_c(\Omega)$, $\theta_{n}=1$ on $K_{n-1}$, $\theta_{n} = 0$ on $K_{n+1}^c$, 
$$\textnormal{Supp } (\nabla^{k} \theta_n) \subseteq K_{n+1}\setminus \textnormal{Int}(K_{n-1}).$$
To prove the theorem it suffice to prove that 
$$u_n:=u \theta_n\xrightarrow[n\longrightarrow +\infty]{}u \text{ in }W^{m,p}(\R^N),$$
because then we can conclude  by using  the density of $\mathscr{C}^\infty_c(\Omega)$ into $W^{m,p}(\textnormal{Int}(K_{n+2}))$, and a diagonal argument.
Let $k\leqslant m$ be a positive integer. To prove the claim we first estimate the $L^p$ norm :
$$\|u_n-u\|_{L^p(\R^N)}^p \leqslant \int_{\overline{\Omega} \backslash K_{n-1}}\vert u \vert^p\;dx.$$
Using the fact that $(\Omega \backslash K_n)_{n\in \mathbb{N}}$ is a decreasing sequence of Lebesgue measurable sets, and thanks to the condition $|\overline{\Omega} \setminus \Omega |=0$, we know that $|\overline{\Omega} \backslash K_n|\longrightarrow 0$ as $n\longrightarrow +\infty$ and therefore $u_n \longrightarrow u$ in $L^p(\R^N)$. Next for the norm of gradients we will use a covering of  $\partial \Omega$. More precisely, the infinite family $(B(x,2^{-(n-2)}))_{x\in \partial\Omega}$ is a cover of $\textnormal{Supp }(\nabla^k \theta_n)$ and by the famous \emph{5B}-covering lemma (see for instance \cite[Theorem 2.2.3]{at}) there exists a countably subcover indexed by $(x_i)_{i\in \mathbb{N}} \subseteq \partial \Omega$ such that $(B(x_i,2^{-(n-2)}))_{i\in \mathbb{N}}$ is a disjoint family,
$$ \textnormal{Supp }(\nabla^k\theta_n) \subseteq \bigcup_{i\in \mathbb{N}} B(x_i, 5\cdot 2^{-(n-2)}), \textnormal{ and } \sum_{i\in \mathbb{N}} \boldsymbol{1}_{B(x_i,5\cdot2^{-(n-2)})} \leqslant N_0,$$
for a universal constant $N_0 \in \mathbb{N}$. In the sequel, we simply write $B_n(x_i)$ instead of $B(x_i,5\cdot2^{-(n-2)})$. Afterwards, we estimate
\begin{align*}
\|\nabla^k u_n-\nabla^k u\|_{L^p(\mathbb{R}^N)}^p &\leqslant  C\int_{\overline{\Omega} \backslash K_{n-1}}\vert \nabla^k u \vert^p\;dx + C\sum_{\substack{k = |\beta| + |\gamma| \\ \gamma \ne 0}}\int_{\overline{\Omega} \backslash K_{n-1}} |\partial^\beta u|^p \vert \partial^\gamma\theta_n \vert^p\; dx.
\end{align*}
The first term tends to $0$ as $n \longrightarrow +\infty$ for the same reasons as before. For the other term we use the following estimate
$$
\vert \partial^\gamma \theta_n(x)\vert^p \leqslant  \varepsilon_n^{-pN} \int_{K_n}\varepsilon_n^{-p\vert \gamma\vert}\left\vert\partial^\gamma\rho\left(\frac{x-y}{\varepsilon_n}\right)\right\vert^pdy \leqslant C\varepsilon_n^{-p\vert \gamma\vert}.
$$
The function $u$ vanishes almost everywhere on the open set $\overline{\Omega}^c$, so $\partial^\beta u$ is zero in $\mathscr{D}'(\overline{\Omega}^c)$ and vanishes almost everywhere on this open set.  Hence  the Poincar\'e inequality \eqref{poincare} applies to all  the $\partial^{\beta}u$ for $|\beta|<m$, and for all ball $B_n(x_i)$ such that $2^{-(n-2)} \leqslant r_0$, thanks to our capacitary condition \eqref{cond} we get
$$ {\rm Cap}_{1,p}(x_i+5^{-1}\cdot2^{n-2}(Z(\partial^\beta u)-x_i)) \geqslant C\varepsilon_n^{-(N-p)} {\rm Cap}_{1,p}(\overline{\Omega}^c \cap B(x_i,5\cdot2^{2-n})) \geqslant C\delta_0.$$
Therefore,
\begin{eqnarray}
\int_{B_n(x_i)} |\partial^\beta u|^p\; dx \leqslant  C \delta_0^{-1}\varepsilon_n^p \int_{B_n(x_i)} |\nabla \partial^\beta u|^p\; dx,
\end{eqnarray}
and using successively $(k- \vert \beta\vert)-$times the Poincar\'e inequality and the covering of $\partial\Omega$, we get
\begin{align*}
    \int_{\overline{\Omega} \backslash K_{n-1}} |\partial^\beta u|^p \vert \partial^\gamma \theta_n \vert^p\; dx &\leqslant C\varepsilon_n^{-p\vert \gamma \vert}\int_{\overline{\Omega} \backslash K_{n-1}} |\partial^\beta u|^p\; dx \\
    &\leqslant C\varepsilon_n^{-p\vert \gamma \vert} \sum_{i\in \mathbb{N}}\int_{B_n(x_i)} |\partial^\beta u|^p\; dx \\
    &\leqslant C \sum_{i\in \mathbb{N}}\int_{B_n(x_i)} |\nabla^k u|^p\; dx \\
    &\leqslant CN_0\int_{\overline{\Omega} \backslash K_{n-5}} |\nabla^k u|^p\; dx 
\end{align*}
and this tends to zero as $n \longrightarrow +\infty$ so follows the proof.
\end{proof}


\section{Examples of domains satisfying our condition}

As we said in the introduction, any smooth enough domain will satisfy our condition. For instance domains satisfying an external corkscrew condition as defined below.\newline

\begin{definition}\label{corkscrew} Let $\Omega\subset \R^N$ be an open and bounded set, $a\in (0,1)$, and $r_0>0$. We say that $\Omega$ satisfies an $(a,r_0)-$external corkscrew condition if for every $x \in \partial \Omega$ and $r\leqslant r_0$, one can find a ball $B(y, ar )$ such that 
$$B(y,ar) \subset B(x,r) \cap \overline{\Omega}^c.$$
\end{definition}\leavevmode

We give a non-exhaustive list of class of domains includes in $\mathscr{O}(D)$ :\newline

\begin{itemize}
    \item $\mathscr{O}_{\textnormal{convex}}(D):=\left\{ \Omega \subseteq D \; \middle| \; \Omega \textnormal{ open and convex} \right\}$.
      \item $\mathscr{O}_{\textnormal{seg}}^{r_0}(D):=\left\{ \Omega \subseteq D \; \middle| \; \Omega \textnormal{ open and has the } r_0\textnormal{-external segment property} \right\}$, we say $\Omega$ has the $r_0$-external segment property if for every $x\in \partial\Omega$, there exists a vector $y_x \in \mathbb{S}^{N-1}(0,r_0)$ such that $x+ty_x \in \Omega^c$ for $t\in(0,1)$. This notion can also be generalized by the ``flat cone'' condition as in \cite[Definition 5.2]{BuZo94} (see also \cite{bz2}).
          \item $\mathscr{O}_{\textnormal{Lip}}^\lambda(D):=\left\{ \Omega \subseteq D \; \middle| \; \Omega \textnormal{ open and is a Lipschitz domain} \right\}$.
       \item $\mathscr{O}_{\textnormal{Reif flat}}^{\delta_0, r_0}(D):=\left\{ \Omega \subseteq D \; \middle| \; \Omega \textnormal{ open and is }(\varepsilon_0, \delta_0)-\textnormal{Reifenberg flat} \right\}$, we say $\Omega$ is $(\varepsilon_0, \delta_0)-\textnormal{Reifenberg flat}$ for $\varepsilon_0 \in(0,1/2)$ and $\delta_0 \in(0,1)$ if for all $x\in \partial \Omega$ and $ \delta \in (0, \delta_0]$, there exists an hyperplan $\mathscr{P}_x(\delta)$ of $\R^N$ such that $x\in \mathscr{P}_x(\delta)$ and 
    \begin{align*}
        d_H\left(\partial \Omega \cap \overline{B}(x,\delta), \mathscr{P}_x(\delta) \cap \overline{B}(x,\delta) \right) \leqslant \delta\varepsilon_0.
    \end{align*}
    Moreover for all $x\in \partial \Omega$, the set
    \begin{align*}
        B(x,\delta_0) \cap \left\{ x \in \R^N \; \middle| \; d(x, \mathscr{P}_x(\delta_0))\geqslant2\delta_0 \varepsilon_0\right\}
    \end{align*}
    have two connex components ; one is include in $\Omega$, the other one in $\R^N \backslash \Omega$.
      \item $\mathscr{O}^{\varepsilon}_{\textnormal{cone}}(D):=\left\{ \Omega \subseteq D \; \middle| \; \Omega \textnormal{ open and has the external }\varepsilon-\textnormal{cone condition} \right\}$, we say $\Omega$ has the external $\varepsilon-$cone condition if there exists a cone $C$ of angle $\varepsilon$ such that for every $x\in \partial \Omega$, there exists a cone $C_x$ congruent to $C$ by rigid motion and such that $x$ is the vertex of $C_x$ and $C_x \subset \Omega^c$.
    \item $\mathscr{O}^{a,r_0}_{\textnormal{corks}}(D):=\left\{ \Omega \subseteq D \; \middle| \; \Omega \textnormal{ open and has the } (a,r_0)-\textnormal{external corkscrew condition} \right\}$, see definition \ref{corkscrew}.
    \item $\mathscr{O}^{\delta_0,r_0}_\textnormal{cap}(D):=\left\{ \Omega \subseteq D \; \middle| \; \Omega \textnormal{ open and has the }(\delta_0,r_0)-\textnormal{capacity condition \ref{cond}} \right\}$.
\end{itemize}
$$ $$
It is easy to see that for some fixed parameters we have the inclusions $$\mathscr{O}^{\varepsilon}_{\textnormal{cone}}(D)\subseteq \mathscr{O}^{a,r_0}_{\textnormal{corks}}(D),$$
and
$$\mathscr{O}_{\textnormal{convex}}(D) \subseteq\mathscr{O}_{\textnormal{Lip}}^\lambda(D)\subseteq\mathscr{O}_{\textnormal{Reif flat}}^{ \delta_0,r_0}(D)\subseteq \mathscr{O}_{\textnormal{corks}}^{a,r_1}(D)\subseteq\mathscr{O}^{\delta_1,r_2}_\textnormal{cap}(D).$$
A segment is a locally Lipschitz manifold of dimension $1$, the properties of $(1,p)-$capacity implies that in dimension $N<p+1$ we have $$\mathscr{O}_{\textnormal{seg}}^{r_0}(D) \subseteq \mathscr{O}^{\delta_0,r_1}_\textnormal{cap}(D).$$

Any $C^1$ domain or Lipschitz domain satisfies an external corkscrew condition. It also follows from porosity estimates that the crokscrew condition implies $|\overline{\Omega} \setminus \Omega|=0$, as stated in the following useful proposition.\\

\begin{proposition}\label{measureboundary}
    If $\Omega \in \mathscr{O}_{\textnormal{corks}}^{a,r_0}(D)$, then $\vert \partial \Omega \vert =0$.
\end{proposition}\leavevmode

 \begin{proof}For all $x\in \partial \Omega$ and $r\leqslant r_0$, there exists $y\in\R^N$ such that $$B(y,ar)\subset B(x,r) \cap \overline{\Omega}^c \subset\R^N\backslash\partial \Omega.$$
 In other words, $\partial \Omega$ is a $\sigma-$porus set in $\R^N$, in the sense of \cite[Definition 2.22]{z}, with $\sigma=2a$.   In virtue of \cite{z} (see the last paragraph at the bottom of page 321 in \cite{z}, or see also \cite[Proposition 3.5]{preiss2}), we conclude $\vert \partial \Omega \vert =0$. 
 \end{proof}\leavevmode

\begin{corollary}If $\Omega$ lies in one of the following classes :   $\mathscr{O}^{\varepsilon}_{\textnormal{cone}}(D)$, $\mathscr{O}_{\textnormal{convex}}(D)$, $\mathscr{O}_{\textnormal{Lip}}^\lambda(D)$, $\mathscr{O}_{\textnormal{Reif flat}}^{ \delta_0,r_0}(D)$, or $\mathscr{O}_{\textnormal{corks}}^{a,r_0}(D)$, then $\Omega$ is $(m,p)-$stable for any $m\geqslant 1$ and $1\leqslant p <+\infty$.
\end{corollary}\leavevmode


\section{Stability with respect to domain perturbation}
\label{stability}

As before, we consider a fixed bounded domain $D\subset \R^N$. Let $\Omega$ and $(\Omega_n)_{n\in \mathbb{N}}$ be bounded subdomains of $D$ such that $\overline{\Omega_n} \longrightarrow \overline{\Omega}$ and $\overline{D} \backslash \Omega \longrightarrow \overline{D} \backslash \Omega $ as $n \longrightarrow +\infty$ for the Hausdorff convergence. In particular, this implies the compact convergence of the sequence. In this section we verify that the $(m,2)-$stability of $\Omega$ implies the Mosco convergence of the sequence $(H^m_0(\Omega_n))_{n\in \mathbb{N}}$ towards $H^m_0(\Omega)$. This will follows from the same argument as for the classical case of $H^1_0$, but for the sake of completeness we give here the full details. For this purpose, we first prove the equivalence between $\gamma_m-$convergence and Mosco convergence (Proposition \ref{equiv1}). Then we show that $(\Omega_n)_{n\in \mathbb{N}}$ $\gamma_m-$converges to $\Omega$ according to $(m,2)-$stability of $\Omega$ (Proposition~\ref{main2}).  \newline

\begin{definition}\label{gammam}
    The sequence $(\Omega_n)_{n\in\mathbb{N}}$ $\gamma_m-$converges to $\Omega$ if for all $f \in L^2(D)$, the sequence $(u_n)_{n\in \mathbb{N}}$ strongly converges in $H^m_0(D)$ to $u$, where $u_n$ (\textit{resp.} $u$) is the unique solution of the Dirichlet problem $(-\Delta)^mu_n =f$ (\textit{resp.} $(-\Delta)^mu=f$) in $H^m_0(\Omega_n)$ (\textit{resp.} $H^m_0(\Omega)$).
\end{definition}\leavevmode

\begin{definition} The sequence $(H^m_0(\Omega_n))_{n \in \mathbb{N}}$ converges to $H^m_0(\Omega)$ in the sense of Mosco if the following holds :
   \begin{enumerate}
       \item If $(v_{n_k})_{k\in\mathbb{N}}$ is a subsequence, where $v_{n_k}\in H^m_0(\Omega_{n_k})$, and weakly converges to $v\in H^m_0(D)$, then $v\in H^m_0(\Omega)$.
       \item For all $v\in H^m_0(\Omega)$, there exists a sequence $(v_n)_{n\in\mathbb{N}}$, where $v_n \in H^m_0(\Omega_n)$, which strongly converges to $v$ in $H^m_0(D)$.
   \end{enumerate}
\end{definition}\leavevmode

\begin{proposition}\label{equiv1}
    The sequence $(\Omega_n)_{n\in\mathbb{N}}$ $\gamma_m$-converges to $\Omega$ if, and only if, $(H^m_0(\Omega_n))_{n\in \mathbb{N}}$ converges to $H^m_0(\Omega)$ in the sense of Mosco.
\end{proposition}\leavevmode
\begin{proof}
    Suppose that the sequence $(\Omega_n)_{n\in\mathbb{N}}$ $\gamma_m$-converge to $\Omega$. Let $(v_{n_k})_{k\in\mathbb{N}}$ be a subsequence which weakly converges to $v\in H^m_0(D)$, where $v_{n_k}\in H^m_0(\Omega_{n_k})$. Consider the $L^2(D)$ function $f:= (-\Delta)^m v$ in such a way that $v$ is the unique solution of the Dirichlet problem in $D$. The $\gamma_m-$convergence implies that $(u_{n_k})_{k \in \mathbb{N}}$ strongly converges in $H^m_0(D)$ to $u \in H^m_0(\Omega)$ where $u_{n_k}$ satisfies $(-\Delta)^m u_{n_k} =f$ in $\Omega_{n_k}$ and $u$ satisfies $(-\Delta)^mu= f$ in $\Omega$. It suffices to show that $u = v$. For all $k \in \mathbb{N}$,
    \begin{align*}
        \int_{D}\nabla^m u_{n_k} :  \nabla^m(u_{n_k}-v_{n_k})\; dx &= \int_{\Omega_{n_k}}\nabla^m u_{n_k} : \nabla^m(u_{n_k}-v_{n_k})\; dx \\&=  \int_{\Omega_{n_k}}f(u_{n_k}-v_{n_k})\; dx \\
        &=  \int_{D}f(u_{n_k}-v_{n_k})\; dx,
    \end{align*}
    and as $k \longrightarrow +\infty$, we use the strong convergence of $(u_{n_k})_{k\in \mathbb{N}}$ and the weak convergence of $(v_{n_k})_{k\in \mathbb{N}}$ to get
    \begin{align*}
         \int_{D} \nabla^mu: \nabla^m(u-v)\; dx = \int_D f(u-v) \; dx.
    \end{align*}
    Then according to equality $f:=(-\Delta)^m v$, we have
    \begin{align*}
        \int_D \vert \nabla^m (u - v)\vert^2 dx &= \int_D \nabla^m(u - v) : \nabla^mu \;dx - \int_D \nabla^m(u- v) : \nabla^mv \;dx \\
        &= \int_D f(u-v) \; dx -\int_D(u - v) (-\Delta)^mv \;dx =0.
    \end{align*}
    The domain $D$ is bounded and due to the classical Poincaré inequality, the functions $u$ and $v$ are equals and the first point of Mosco convergence follows. To prove the second one, consider $v\in H^m_0(\Omega) \subseteq H^m_0(D)$ and let $f := (-\Delta)^m v$. The same arguments used before implies that $(u_n)_{n \in \mathbb{N}}$ strongly converges in $H^m_0(D)$ to $u = v$ where $u_n \in H^m_0(\Omega_n)$ is the unique solution of Dirichlet problem with respect to $f$. \newline
    
   Now suppose that $(H^m_0(\Omega_n))_{n\in \mathbb{N}}$ converges in the sense of Mosco to $H^m_0(\Omega)$. Consider $f\in L^2(D)$ and the associated solutions $u_n$ of the Dirichlet problem in $\Omega_n$. For all $n \in \mathbb{N}$,
    \begin{align*}
    \int_D \vert \nabla^m u_n\vert^2 dx = \int_{\Omega_n}\nabla^m u_n: \nabla^m u_n \; dx = \int_{\Omega_n} f u_n \; dx = \int_D fu_n\; dx,
    \end{align*}
    we infer that the sequence $(u_n)_{n\in \mathbb{N}}$ is bounded in $H^m_0(D)$ since
    \begin{align*}
        \Vert f \Vert_{L^2(D)}\Vert u_n\Vert_{L^2(D)} \leqslant \Vert f \Vert_{L^2(D)}\Vert u_n\Vert_{H^m_0(D)}.
    \end{align*}
   Let $(u_{n_k})_{k \in \mathbb{N}}$ be a subsequence which weakly converges to a function $v\in H^m_0(D)$. Using the Mosco convergence, $v\in H^m_0(\Omega)$ and for all $\varphi \in H^m_0(\Omega)$, there exists a sequence  $(\varphi_k)_{k \in \mathbb{N}}$, with $\varphi_k \in H^m_0(\Omega_{n_k})$, strongly converging to $\varphi$ in $H^m_0(D)$. Hence, for all $k \in \mathbb{N}$,
    \begin{align*}
        \int_D \nabla^m u_{n_k}: \nabla^m \varphi_k \;dx =  \int_{\Omega_{n_k}} \nabla^m u_{n_k}: \nabla^m \varphi_k \;dx = \int_{\Omega_{n_k}}f \varphi_k \;dx = \int_D f \varphi_k \;dx,
    \end{align*}
   and using the strong convergence of $(\varphi_k)_{k \in \mathbb{N}}$ and the weak convergence of $(u_{n_k})_{n\in \mathbb{N}}$ as $k \longrightarrow + \infty$, we obtain
    \begin{align*}
        \int_\Omega\nabla^m v :\nabla^m \varphi \; dx =\int_D \nabla^m  v :\nabla^m  \varphi \; dx = \int_D f \varphi \;dx = \int_\Omega f \varphi \; dx.
    \end{align*}
    The uniqueness of the solution of the Dirichlet problem proves $u = v$. Moreover,
\begin{align*}
    \int_{D}\vert \nabla^m u_{n_k}\vert^2 dx = \int_{\Omega_n}fu_{n_k}\;dx =  \int_{D}fu_{n_k} \;dx,
\end{align*}
and
\begin{align*}
      \int_{D}fu_{n_k}\;dx \xrightarrow[k \rightarrow + \infty]{} \int_D fu \;dx =   \int_{D}\vert \nabla^m u \vert^2 dx.
\end{align*}
This yields
\begin{align*}
    \Vert u_{n_k} \Vert_{H^m_0(D)}\xrightarrow[k \rightarrow + \infty]{}   \Vert u \Vert_{H^m_0(D)},
\end{align*}
and the convergence of the subsequence is strong. By uniqueness of the limit, the whole sequence is strongly converging to $u$ in $H^m_0(D)$.
\end{proof}\leavevmode

\begin{definition} \label{hausdorff}For two closed sets $A,B\subset \R^N$, the Hausdorff distance $d_H(A,B)$ is defined by 
$$d_H(A,B):=\max_{x\in A} {\rm dist}(x,B) + \max_{x\in B} {\rm dist}(x,A).$$
A sequence of closed sets $(A_n)_{n\in \mathbb{N}}$ converges to $A$ for the Hausdorff distance if $d_H(A_n,A) \longrightarrow 0$ as $n \longrightarrow +\infty$. In this case, we will write  $A_n \overset{d_H}{\longrightarrow} A$.
\end{definition}\leavevmode

Next, we define the complementary Hausdorff distance over $\mathscr{O}(D)$ by 
$$ d_{H^c}(\Omega_1, \Omega_2) := d_H(\overline{D} \backslash \Omega_1, \overline{D} \backslash \Omega_2),$$
and one can show that the topolgy induced on $\mathscr{O}(D)$ is compact.
In the sequel we will use the following well known result.\newline

\begin{proposition}\label{compacthausdorff} If $( \Omega_n)_{n\in \mathbb{N}}$ is a sequence in $\mathscr{O}(D)$ such that $\Omega_n \overset{d_{H^c}}{\longrightarrow} \Omega \in \mathscr{O}(D)$, then for any compact set $K\subset \Omega$ there exists $n_0\in \mathbb{N}$ depending on $K$ such that $K\subset \Omega_n$ for all $n\geqslant n_0$.
\end{proposition} \leavevmode

\begin{proof}  Since $K$ is compact and $\Omega$ is open, we know that 
$$\inf_{x\in K} {\rm dist}(x,\Omega^c)=:a>0.$$
 By Hausdorff convergence of the complements, there exists $n_0(a)\in \mathbb{N}$ such that for all $n\geqslant n_0(a)$, 
 $$\Omega_n^c \subset  \{y \in \R^N\; | \; {\rm dist}(y,\Omega^c)<a/2\}.$$
We deduce from the triangle inequality  that $\inf_{x\in K} {\rm dist}(x,\Omega_n^c) >0$ for $n$ large enough and in particular $K\subset \Omega_n$.
\end{proof}

We are now ready to state the following result that will directly imply Corollary \ref{main22} written in the introduction.\newline

\begin{proposition} \label{main2} Let $( \Omega_n)_{n\in \mathbb{N}}$ be a sequence in $\mathscr{O}(D)$ such that $\overline{\Omega_n} \overset{d_H}{\longrightarrow} \overline{\Omega}$ and $ \Omega_n \overset{d_{H^c}}{\longrightarrow}  \Omega$ where $\Omega \in \mathscr{O}(D)$. If  $\Omega$  is $(m,2)-$stable, then the sequence $(\Omega_n)_{n\in \mathbb{N}}$ $\gamma_m-$converges to $\Omega$ or equivalently, $(H^m_0(\Omega_n))_{n\in \mathbb{N}}$ converges to $H^m_0(\Omega)$ in the sense of Mosco.
\end{proposition}\leavevmode

\begin{proof}[Proof of  Proposition \ref{main2}]
    Consider $f \in L^2(D)$. We know that the sequence $(u_n)_{n \in \mathbb{N}}$ of the Dirichlet problem solutions associated to $f$ is bounded in $H^{m}_0$. There exists a subsequence $(u_{n_k})_{k\in \mathbb{N}}$ which weakly converges to a function $v \in H^m_0(D)$. Let $\varphi \in \mathscr{C}_c^\infty(\Omega)$ be test function. By complementary Haussdorf convergence, there exists an integer $k_0 \in \mathbb{N}$ such that for all $k\geqslant k_0$,
    \begin{align*}
        \textnormal{Supp}(\varphi) \subseteq \Omega_{n_k}.
    \end{align*}
  Thus, for all $k\geqslant k_0$,
    \begin{align*}
        \int_{\Omega} \nabla^m u_{n_k} : \nabla^m \varphi \;dx =\int_{\Omega_{n_k}} \nabla^m u_{n_k} : \nabla^m \varphi \;dx = \int_{\Omega_{n_k} }f\varphi \; dx  =\int_{\Omega }f\varphi \; dx,
    \end{align*}
    and by weak convergence of $(u_{n_k})_{k \geqslant k_0}$ in $H^m_0(D) \supseteq H^m_0(\Omega)$, 
   \begin{align*}
        \int_{\Omega} \nabla^m v : \nabla^m \varphi \;dx  =\int_{\Omega }f\varphi \; dx.
    \end{align*}
  Thanks to the uniqueness of the Dirichlet problem, it suffices to prove that $v\in H^m_0(\Omega)$. In this case, $v = u$ and the whole sequence $(u_n)_{n \in \mathbb{N}}$ strongly converge to $u$. Up to a subsequence, we can assume that $(u_{n_k})_{k \in \mathbb{N}}$ converges almost everywhere to $v$. The functions $u_{n_k}$ vanishes $(m,2)-$quasi everywhere on $\overline{\Omega}_{n_k}^c$ so almost everywhere. By Hausdorff convergence of the adherence we know that for all compact $K\subset \overline{\Omega}^c$ then $K\subset \overline{\Omega}_n^c$ for $n$ large enough thus finaly $v =0$ almost everywhere in $\overline{\Omega}^c$. Using the definition of $(m,2)-$stability, we conclude $v \in H^m_0(\Omega)$.
\end{proof}

\section{Proof of Theorem \ref{main3}}

 In this section we give a proof of Theorem \ref{main3} stated in the introduction. \newline
 
\begin{proof}[Proof of Theorem \ref{main3}] Let $\Omega, \Omega_n \subset D$ be  bounded domains as in the statement of Theorem~\ref{main3} that satisfies
$$\Omega_n  \overset{d_{H^c}}{\longrightarrow}  \Omega,$$
and such that \eqref{cond}  holds true for all $\Omega_n$ with the same $\delta_0>0$ and $r_0>0$. We want to prove that $\Omega_n$ $\gamma_m-$converges  to $\Omega$. To this aim we start with a similar argument as in the proof of  Proposition \ref{main2}. Consider $f \in L^2(D)$. We know that the sequence $(u_n)_{n \in \mathbb{N}}$ of the Dirichlet problem solutions associated to $f$ in $\Omega_n$ is bounded in $H^{m}_0(D)$. There exists a subsequence $(u_{n_k})_{k\in \mathbb{N}}$ which weakly converges to a function $v \in H^m_0(D)$. Let $\varphi \in \mathscr{C}_c^\infty(\Omega)$ be test function. By complementary Hausdorff convergence, there exists an integer $k_0 \in \mathbb{N}$ such that for all $k\geqslant k_0$,
    \begin{align*}
        \textnormal{Supp}(\varphi) \subseteq \Omega_{n_k}.
    \end{align*}
  Thus, for all $k\geqslant k_0$,
    \begin{align*}
        \int_{\Omega} \nabla^m u_{n_k} : \nabla^m \varphi \;dx =\int_{\Omega_{n_k}} \nabla^m u_{n_k} : \nabla^m \varphi \;dx = \int_{\Omega_{n_k} }f\varphi \; dx  =\int_{\Omega }f\varphi \; dx,
    \end{align*}
    and by weak convergence of $(u_{n_k})_{k \geqslant k_0}$ in $H^m_0(D) \supseteq H^m_0(\Omega)$, 
   \begin{align*}
        \int_{\Omega} \nabla^m v : \nabla^m \varphi \;dx  =\int_{\Omega }f\varphi \; dx.
    \end{align*}
  Now thanks to the uniqueness of the Dirichlet problem, it suffices to prove that $v\in H^m_0(\Omega)$. In section \ref{sec7} we prove that the class $\mathscr{O}_\textnormal{cap}^{\delta_0, r_0}(D)$ is compact for the complementary Hausdorff convergence. Moreover, since $\Omega$ satisfies the $(\delta_0,r_0)$-capacitary condition we know from Theorem \ref{main}  that $\Omega$ is a $(m,2)-$stable domain. Thus in order to conclude the proof we are left to prove that $v=0$ a.e. in $\overline{\Omega}^c$. From here the proof differs from the one of Theorem \ref{main} because  we do not know anymore  that  $\overline{\Omega_n} \overset{d_H}{\longrightarrow} \overline{\Omega}$. Instead, we shall benefit from the fact that \eqref{cond} holds true for the whole sequence $\Omega_n$ and we will use a construction similar to the one used in the proof of Theorem \ref{main}, but on the functions $u_n$. From now on we will simply denote by $n$ instead of $n_k$ for the subsequence $u_n\rightarrow v$ in $H^m(D)$ as $n \longrightarrow+ \infty$. Let $K\subset \overline{\Omega}^c$ be an arbitrary compact set and let $\varepsilon>0$ be given. Our goal is to prove that $v=0$ a.e. on $K$. For a general closed set $F\subset \R^N$ and $\lambda>0$ we denote by $(F)_\lambda$ the $\lambda-$enlargement of $F$, namely,
$$F_{\lambda}:= \left\{x\in \R^N \;\middle| \; {\rm dist}(x,F)\leqslant \lambda\right\}.$$
By the Hausdorff convergence of $\Omega_n^c$ to $\Omega^c$ we know that there exists $n_0(\varepsilon)\in \mathbb{N}$ such that for all $n\geqslant n_0(\varepsilon)$,
$$\Omega^c \subset (\Omega_n^c)_\varepsilon, \text{ and } \Omega_n^c\subset (\Omega^c)_\varepsilon.$$
From the above we deduce that 
\begin{eqnarray}
K\subset \Omega^c \subset (\Omega_n^c)_\varepsilon \subset (\Omega^c)_{2\varepsilon}. \label{inclusion}
\end{eqnarray}
Next, we want to construct a test function in $\mathscr{C}^\infty_c(\Omega_n)$ which is very close to $u_n$ in $L^2$ and equal to $0$ on $K$. Let us consider the following subset of $\Omega_n$, 
$$ A_{n,\varepsilon}  := \left\{ x \in \Omega_n \; \middle| \; d(x,  \Omega_n^c)\geqslant   10 \varepsilon \right\},$$
and the function 
$$w_{n,\varepsilon}:=u_n {\bf 1}_{A_{n,\varepsilon}}.$$
The main point being that $w_{n,\varepsilon}=0 \text{ in } (\Omega_n^c)_\varepsilon$ and in virtue of \eqref{inclusion} we deduce that $w_{n,\varepsilon}=0$ on $K$. Now we estimate the difference $w_{n,\varepsilon}- u_n$ in $L^2(\R^N)$  using a covering of  $\partial \Omega_n$. More precisely, the infinite family $(B(x,20 \varepsilon))_{x\in \partial\Omega_n}$ is a cover of $\Omega_n \setminus A_{n,\varepsilon}$ and by the  \emph{5B}-covering Lemma  there exists a countably subcover indexed by $(x_i)_{i\in \mathbb{N}} \subset \partial \Omega$ such that $(B(x_i,20\varepsilon))_{i\in \mathbb{N}}$ is a disjoint family,
$$ \Omega_n \setminus A_{n,\varepsilon} \subset\bigcup_{i\in \mathbb{N}} B(x_i, 100\varepsilon), \textnormal{ and } \sum_{i\in \mathbb{N}} \boldsymbol{1}_{B(x_i,100\varepsilon)} \leqslant N_0,$$
for a universal constant $N_0 \in \mathbb{N}$. Then we can estimate, 
$$\int_{D} |w_{n,\varepsilon}-u_n|^2\;dx \leqslant \int_{\Omega_n \setminus A_{n,\varepsilon}} \vert u_n\vert ^2\; dx.$$
The functions $\partial^\beta u_n$ vanishes almost everywhere on the open set $\overline{\Omega_n}^c$, so thanks to our capacitary condition \eqref{cond} we have for $\varepsilon$ small enough
 $$(100 \varepsilon)^{-(N-p)}{\rm Cap}_{1,2}(Z(\partial^\beta u_n)) \geqslant C\varepsilon^{-(N-p)} {\rm Cap}_{1,2}(\overline{\Omega_n}^c \cap B(0,100\varepsilon)) \geqslant C\delta_0.$$
Therefore, the Poincar\'e inequality \eqref{poincare} applies to $\partial^\beta u_n$ in all ball $B(x_i, 100\varepsilon)$ gives
\begin{eqnarray}
\int_{B(x_i,100\varepsilon)} \vert \partial^\beta u_n\vert^2 \; dx \leqslant   C \delta_0^{-1}\varepsilon^2\int_{B(x_i,100\varepsilon)} \vert \nabla  \partial^\beta u_n\vert^2 \; dx.
\end{eqnarray}
We deduce that 
\begin{align*}
   \int_{\Omega_n \setminus A_{n,\varepsilon}} \vert u_n\vert^2 \; dx     &\leqslant  \sum_{i\in \mathbb{N}}\int_{B(x_i,100\varepsilon)} \vert u_n\vert^2 \; dx \\
    &\leqslant C \sum_{i\in \mathbb{N}} \varepsilon^{2m}\int_{B(x_i,100\varepsilon)} |\nabla^m u_n|^2\; dx \\
    &\leqslant CN_0  \varepsilon^{2m}\int_{D} |\nabla^m u_n|^2\; dx \\
    & \leqslant C \varepsilon^2,
\end{align*}
because the sequence $u_n$ is uniformly bounded in $H^1(D)$. In conclusion we have proved the following : for each $\varepsilon>0$, we have $n_0(\varepsilon) \in \mathbb{N}$ such that for all $n \geqslant n_0(\varepsilon)$, there exists 
$w_{n, \varepsilon} \in L^2(D)$ such that $\|w_{n, \varepsilon}-u_n\|_{L^2}\leqslant C\varepsilon$   and $w_{n, \varepsilon}=0$ on $K$. Now for $n$ great enough let $\varepsilon= 2^{-n}$ and let $w_n:=w_{n_0(2^{-n}), 2^{-n}}$. We can assume that $n_0(2^{-n})\rightarrow+\infty$. The function $w_n$ converges to $v$ in $L^2$ because $u_n$ converges to $v$ in $L^2$, and $w_n=0$ on $K$ for all $n \in \mathbb{N}$. Therefore, up to a subsequence, $w_n$ converges a.e. on $K$ and this shows that $u=0$ a.e. on $K$. Since $K$ is arbitrary, this shows that $v=0$ a.e. on $\overline{\Omega}^c$, hence $u\in H^m_0(\Omega)$ because $\Omega$ is $(m,2)-$stable. This achieves the proof.
\end{proof}

\section{Existence for shape optimisation problems under geometrical constraints}\label{sec7}

Let $D \subset \mathbb{R}^N$ be a fixed bounded open set and let $\mathscr{O}_D:=\left\{\Omega \subseteq D \; \middle| \; \Omega \textnormal{ is open} \right\}$ denote all open subsets of $D$. For a shape functional $F : \mathscr{O}_D \longrightarrow \mathbb{R}^+$, it is a natural question to ask if there exists extremal points. In order to answer this question, we introduce a subfamily of $ \mathscr{O}_D $ which is compact for the $\gamma_m-$convergence and satisfies the capacitary condition (\ref{cond}). If $F$ is lower semi-continuous for the $\gamma_m-$convergence, then we use Theorem \ref{main3} to conclude.

The existence of minimizers for shape functionals has been studied in \cite{BuZo94}. They showed that the class of open subset satisfying the $(r_0,\delta_0)-$capacitary condition
$$ \forall x \in \partial \Omega, \forall r < r_0, \quad \frac{\textnormal{Cap}_{1,2}(\Omega^c \cap B(x,r),B(x,2r))}{\textnormal{Cap}_{1,2}(B(x,r),B(x,2r))}\geqslant \delta_0,$$
is compact for $d_{H^c}$ where for any compact $K$,
$$\textnormal{Cap}_{1,2}(K \cap B(x,r),B(x,2r)):= \inf\Vert \varphi \Vert_{H^1}^2 $$
and the infimum is taken over all $\varphi \in \mathscr{C}^\infty_c(B(x,2r))$ such that $\varphi \geqslant 1 \textnormal{ on }K \cap B(x,r)$. This condition is  weaker than \eqref{cond} due to the fact we consider  $\overline{\Omega}^c$ instead of $\Omega^c$. Therefore, the results in  \cite{BuZo94} implies Theorem \ref{shapeopt} in the case of $m=1$. \newline

\begin{proposition}\label{classcompacity}
    Any class of the following list is compact for the complementary Hausdorff convergence : $\mathscr{O}_{\textnormal{convex}}(D)$,  $\mathscr{O}^{\varepsilon}_{\textnormal{cone}}(D)$, $\mathscr{O}_{\textnormal{corks}}^{a,r_0}(D)$.
\end{proposition}\leavevmode
\begin{proof}
\begin{enumerate}
    \item \textit{Case} $\mathscr{O}=\mathscr{O}_{\textnormal{convex}}(D),\mathscr{O}^{\varepsilon}_{\textnormal{cone}}(D)$. The proof can be founded in \cite[Proposition $5.1.1$, page $126$]{BuBut05}.
      \item \textit{Case} $\mathscr{O}=\mathscr{O}_{\textnormal{corks}}^{a,r_0}(D)$.  Suppose $(\Omega_n)_{n\in \mathbb{N}}$ is sequence of $(a,r_0)-$corkscrew domains which converges to an open set $\Omega \subset D$. Let $x\in \partial \Omega$ and $r\leqslant r_0$. By Hausdorff complementary convergence properties, there exists a sequence $(x_n)_{n \in \mathbb{N}}$ such that $x_n \in \partial \Omega_n$ and $x_n\longrightarrow x$ as $n\longrightarrow +\infty$. By corkscrew conditions, one finds $B(y_n,ar) \subset \overline{\Omega}_n^c\cap B(x_n,r)$ with, up to a subsequence, $y_n \longrightarrow y$ as   $n\longrightarrow +\infty$.  First of all, it is obvious that $B(y,ar) \subset B(x,r)$, it remains to prove that $B(y,ar) \subset \overline{\Omega}^c$. Let $\varepsilon>0$, from the enlargement characterisation of Hausdorff convergence, there exists $N(\varepsilon)\in \mathbb{N}$ such that for every $n \geqslant N(\varepsilon)$,  $\overline{D}\backslash \Omega_n \subset (\overline{D}\backslash \Omega)_\varepsilon$ where
      $$(\overline{D}\backslash \Omega)_\varepsilon := \left\{x\in \R^N \;\middle| \; {\rm dist}(x,\overline{D}\backslash \Omega)\leqslant \varepsilon\right\}.$$
    Thus $B(y_n,ar) \subset(\overline{D}\backslash \Omega)_\varepsilon  $ and passing to the limit as $n \rightarrow + \infty$, then taking the intersection in $\varepsilon$, we get 
    $$B(y,ar) \subset \bigcap_{\varepsilon>0}(\overline{D}\backslash \Omega)_\varepsilon = \overline{D}\backslash \Omega. $$
The ball $B(y,ar)$ is open so we conclude $B(y,ar) \subset \overline{\Omega}^c \cap B(x,r)$.
\end{enumerate}
\end{proof}\leavevmode
\begin{proposition}\label{classgammacompacity}
    Any class of the following list is compact for the $\gamma_m-$convergence: $\mathscr{O}_{\textnormal{convex}}(D)$,    $\mathscr{O}^{\varepsilon}_{\textnormal{cone}}(D)$, $\mathscr{O}_{\textnormal{corks}}^{a,r_0}(D)$.
\end{proposition}\leavevmode
\begin{proof}
    Let $\mathscr{O}$ be one of the class of domains listed above and let $(\Omega_n)_{n \in \mathbb{N}}$ be a sequence in $\mathscr{O}$. Because of the compactness of the complementary Hausdorff convergence in $\mathscr{O}(D)$, there exists a subsequence $(\Omega_n)_{n \in \mathbb{N}}$ denoted by the same indices which $d_{H^c}-$converges to $\Omega \in \mathscr{O}(D)$. Using Theorem \ref{main3} and Proposition \ref{measureboundary}, it is sufficient to prove that $\Omega \in \mathscr{O}$, i.e. $\mathscr{O}$ is closed for $d_{H^c}-$convergence. Proposition \ref{classcompacity} concludes the proof.
\end{proof}\leavevmode

\begin{theorem}\label{shapeopt}
    Let $\mathscr{O}$ be a $\gamma_m-$compact class of subset listed in Proposition \ref{classgammacompacity}. Let $F : \mathscr{O} \longrightarrow \overline{\mathbb{R}}$ be a lower semi-continuous functional for the $\gamma_m-$convergence. There exists $\Omega\in \mathscr{O}$ such that 
    $$ F(\Omega) = \inf\left\{F(\omega) \; \middle|\; \omega \in \mathscr{O} \right\}.$$
\end{theorem}\leavevmode
\begin{proof}
    Let $(\Omega_n)_{n\in\mathbb{N}}$ be a minimising sequence in $\mathscr{O}$, i.e. $ F(\Omega_n)$ converges to $\inf\left\{F(\omega) \; \middle|\; \omega \in \mathscr{O} \right\}$ as $n \longrightarrow +\infty$. Using Proposition \ref{classgammacompacity}, up to a subsequence there exists an open set $\Omega \in \mathscr{O}$ such that
    $$ \Omega_n  \xrightarrow[n\longrightarrow +\infty]{\gamma_m} \Omega.$$
Then by lower semi-continuity of the functional we get
    $$F(\Omega) \leqslant \underset{n \rightarrow + \infty}{\liminf}F(\Omega_n) =  \inf\left\{F(\omega) \; \middle|\; \omega \in \mathscr{O} \right\},$$
    which finishes the proof.
\end{proof}


\bibliography{biblio}
\bibliographystyle{plain}

 \end{document}